\documentclass{scrartcl}
\usepackage[latin9]{inputenc}
\usepackage{a4}
\usepackage{amsfonts}
\usepackage{amssymb}
\usepackage{amsmath}
\usepackage{amsthm}
\usepackage{makeidx}
\usepackage{pdflscape}
\usepackage{setspace}
\usepackage{multirow}
\usepackage{tabularx}
\usepackage{float}
\makeindex
\usepackage{nomencl}
\usepackage{geometry}
\usepackage{extarrows}
\usepackage{changepage}
\usepackage{chngcntr} 
\usepackage[pdftex]{graphicx}
\usepackage{tikz}
\usepackage[toc,page]{appendix}
\providecommand{\keywords}[1]{\textbf{\textit{Keywords: }} #1}
\newcommand{\qq}{{\mathbb{Q}}}   
\begin{document}
\pagenumbering{arabic}
%
%
\newtheorem{thm}{Theorem}
\newtheorem{lem}[thm]{Lemma}
%
\theoremstyle{definition}
\newtheorem{defi}{Definition}
%
\title{On the reducibility behaviour of Thue polynomials}
\author{J.\ K\"onig\thanks{Universit\"at W\"urzburg, Emil-Fischer-Str.\ 30, 97074 W\"urzburg, Germany. email: joachim.koenig@mathematik.uni-wuerzburg.de}}
\maketitle
\begin{abstract}
We prove a result about reducibility behaviour of Thue polynomials over the rationals that was conjectured in \cite{M}.
Special cases have been proved e.g. by M\"uller in \cite{M}, Theorem 4.9, and Langmann (\cite{L}, Satz 3.5).\\
The proof uses ramification theory to reduce the assertion to a statement about permutation groups containing an $n$-cycle. This statement is finally proven with the help of the classification of primitive permutation groups containing an $n$-cycle (a result which rests on the classification of finite simple groups).
\end{abstract}

\section{Introduction and statement of the main theorem}
The famous Hilbert irreducibility theorem states that if $K$ is a number field, $f(t,X) \in K(t)[X]$ an irreducible polynomial, then there are infinitely many specializations $t \mapsto t_0 \in K$ such that $f(t_0,X)$ remains irreducible (and one can even demand that the $t_0$ be integers of $K$). A related question is whether there are also infinitely many integer specializations such that $f(t_0,X)$ becomes reducible. This question is linked to Siegel's theorem about integral points of algebraic curves. \\
In many cases one can obtain finiteness results for the set of reducible specializations. This was done e.g. in \cite{M}. We refer to this paper for some background on the role of Siegel's theorem,
as well as for the basics of ramification theory that will be used here.

The goal of this article is to prove the following theorem, which was conjectured in \cite[Conjecture 4.10]{M}:
\begin{thm}
\label{T1}
Let $H(t,X) \in \mathbb{Q}[t,X]$ be a homogeneous polynomial of degree $n$, not divisible by $t$, and not a proper power over $\overline{\mathbb{Q}}$. Then one of the following holds:
\begin{itemize}
\item $H(t_0,X) - 1$ becomes reducible for only finitely many $t_0 \in \mathbb{Z}$
\item $n \in \{2,4\}$.
\end{itemize}
\end{thm}

{\bf Remark:}\\
Note that the second case cannot be excluded, cf. \cite{M}:\\
There the examples $H(t,X) = X^2-dt^2$ and $H(t,X) = -4dX^2(dX^2-t^2)$, $d>1$ a square-free integer, are given. The Galois groups of the polynomials $H(t,X)-1$ over $\overline{\mathbb{Q}}(t)$ are $C_2$ and $D_4$ respectively; and we will recognize this observation again in the course of the proof.


\section{Some results about permutation groups}
\label{permgroups}
To prove Theorem \ref{T1}, we will need to know about the primitive groups containing a cyclic transitive subgroup. These groups have been classified (using the classification of finite simple groups) by Feit (\cite{F}, 4.1) and Jones (\cite{J}), with the following result:
\begin{thm}
\label{feit}
Let $G\leq S_n$ be a primitive group containing a cyclic transitive subgroup. Then one of the following holds:
\begin{itemize}
\item $n=p \in \mathbb{P}$, and $C_p \leq G \leq AGL_1(p)$ (where $AGL_1(p)$ is the symmetric normalizer of $C_p$, of order $p \times (p-1)$).
\item $G = A_n$ (for $n$ odd) or $S_n$.
\item $n = \frac{q^d-1}{q-1}$ with $d \geq 2$ and $q$ a prime power, $PSL_d(q) \leq G \leq P\Gamma L_d(q)$.
\item $n = 11$, $G = PSL_2(11)$ in its action on 11 points.
\item $G = M_{11}$ or $M_{23}$ in the natural action.
\end{itemize}
\end{thm}
For our situation, we need the following lemma, which follows easily from the above classification result:
\begin{lem}
\label{cyclic}
Let $G\leq S_n$ be a primitive group containing a cyclic transitive subgroup, and $N \trianglelefteq G$.
\begin{itemize}
\item If $G/N$  is cyclic of order $k\geq n$, then $k=n$ is prime and $G = C_n$.
\item If $G/N \cong C_{n/2}$, then $G = C_2$ or $G = S_4$.
\end{itemize}
\end{lem}
\begin{proof}
The only cases worth some consideration come from the (not always cyclic) factor $P\Gamma L_d(p^r)/PSL_d(p^r)$. This factor has order $r\cdot(d,p^r-1)$, and $\frac{n}{2} = \frac{p^{rd}-1}{2(p^r-1)} > \frac{1}{2} p^{r(d-1)}$. But for this to be smaller than $r\cdot (d,p^r-1)$, we need $d=2$, $r=1$ and $p\leq3$, which leaves only the groups $PGL_2(3) = S_4$ and $PGL_2(2) = S_3$.
\end{proof}
We now deduce a result that will help prove Theorem \ref{T1}, but may also be interesting itself. We therefore state it in more generality than what will later be needed in the proof of Theorem \ref{T1}.
\begin{thm}
\label{group}
Let $G \le S_n$ be a finite permutation group generated by a cyclic transitive subgroup $\langle\tau\rangle$ and a normal transitive subgroup $N$. Then the following hold:
\begin{itemize}
\item[i)] $|N\cap \langle\tau \rangle|\ge 2$.
\item[ii)] If $|N\cap \langle\tau \rangle|=2$, then $G$ is of the form $((..(C_{p_1}^{k_1}.C_{p_2}^{k_2})...).C_{p_m}^{k_m}).\tilde{G}$,\footnote{We use ``Atlas notation" here to denote by $N.H$ a group with normal subgroup $N$ and corresponding quotient group $H$.} where $\tilde{G} \in \{C_2, S_4\}$, $\prod{p_i} = \frac{n}{2}$ or $\frac{n}{4}$ respectively, $k_i \in \mathbb{N}$ (and the $p_i$ are primes). In particular $G$ does not contain any element of order larger than $n$.
\end{itemize}
\end{thm}
\begin{proof}
i). Assume $N\cap \langle \tau \rangle = \{1\}$.
As $G/N$ is abelian (even cyclic of order $n$), we have $g\tau g^{-1}\tau^{-1}\in N$ for all $g\in G$.
In particular, $\tau^G \subset N\tau$. But also $C_G(\tau) = \langle \tau\rangle$ (an $n$-cycle is self-centralizing even in all of $S_n$), so $\tau^G$ must be of cardinality $|N|$, i.e.\ $\tau^G = N\tau$.

Denote by $G_1$ a point stabilizer in $G$.
By the transitivity of $N$, the stabilizer $N_1:=G_1\cap N$ has index $n$ in $G_1$. Let $g_1,...,g_n$ be a set of coset representatives of $N_1$ in $G_1$. As $g_iN=g_jN$ already implies $g_i^{-1}g_j\in N\cap G_1 = N_1$, the elements $g_1,...,g_n$ are also a set of coset representatives for $N$ in $G$.
In particular, $G_1$ must intersect every coset of $N$, which is impossible, as the coset of $\tau$ consists entirely of fixed point free elements, namely $n$-cycles. This is the desired contradiction.

ii). Assume now that $|N\cap \langle \tau \rangle| = 2$.\\
From the previous lemma it follows that if $G$ is primitive, then $G = C_2$ or $G=S_4$, and these groups certainly contain no element of order larger than the degree of $G$.

Now let $G$ be imprimitive, $P$ be a minimal partition of $\{1,...,n\}$ into $G$-blocks (i.e. the action of a block stabilizer on a single block is primitive) and let $K\trianglelefteq G$ be the kernel of the action of $G$ on the blocks.
Denote by $\tilde{n}$ the number of blocks in $P$ and by $n'$ the length of a block.
So $G/K$ acts transitively on $\tilde{n}$ points, and $n=n'\tilde{n}$.\\
Now $NK/K$ is a transitive normal subgroup of $G/K$, with cyclic factor group 
$\langle \tau\rangle/(\langle \tau\rangle \cap NK)$, which has order a divisor of $\tilde{n}$.\\
Now if $|K/(N \cap K)| < n'$, then $NK/N (\cong K/(N\cap K))$ would be cyclic of order less then $n'$, so $G/NK$ would be cyclic of order larger than $\frac{\tilde{n}}{2}$. But this is impossible because $G/NK = (G/K)/(NK/K)$, where $G/K$ is generated by the transitive normal subgroup $NK/K$ and the $\tilde{n}$-cycle $\langle \tau\rangle K/K$, and we have already seen in i)  
that these two subgroups must intersect at least in a subgroup of order 2.\\
\\
We are going to prove however, that $|K/(N \cap K)| > n'$ is impossible, and $|K/(N \cap K)| = n'$ only if $K$ is an elementary abelian group (of exponent $n'$).\\
\\
So consider the cyclic group $K/(N\cap K)$. By minimality of the partition $P$, the image of a block stabilizer in the action on a single block is primitive; furthermore the image of $K$ in this action is a transitive normal subgroup (as it contains a cyclic transitive subgroup), and therefore this image, let us call it $H$, is also primitive, as the list in Theorem \ref{feit} shows. So $K$ embeds into a direct product of copies of a primitive group $H$ of degree $n'$, containing an $n'$-cycle. (Also, by transitivity the image $H$ is independent of the chosen block, so $K$ is in fact a subdirect product.) 

We will discuss the different possibilities for $H$.

{\it Case 1: $H$ non-solvable.}\\
In this case we obtain (using Theorem \ref{feit} to get the isomorphism types for $H$) that $K$ is the extension of a (solvable) group of exponent at most $\frac{|H|}{|soc(H)|}$ by a direct product of non-abelian simple groups. It is clear that this group cannot have a cyclic factor larger than $\frac{|H|}{|soc(H)|}$, and by lemma \ref{cyclic} this factor is always smaller than $\deg(H)$.\\
That leaves $H=S_4$ or $H \leq AGL_1(p)$.

{\it Case 2: $C_p \leq H \leq AGL_1(p)$.}\\
Here, $K$ is the (split) extension of an abelian group $A$ of exponent at most $p-1$ by an elementary-abelian $p$-group $P$. If this group $K$ had a normal subgroup with cyclic factor group of order at least $p$, then in particular it would have one of index $p$ (as factoring out only from $A$ can never yield cyclic factors larger than $p-1$). 
But then $K'$ is a proper subgroup of $P$.\\
Now assume furthermore that $H \ne C_p$. Then $Z(K)=\{1\}$ (as even the images of the projection to a component have trivial center). Also, as a special case of a classical theorem by Gasch\"utz (\cite{Ga}), $K$ splits over the normal $p$-subgroup $K'$. Let $U$ be a complement. Then $U$ is abelian and contains an element $x$ of order $p$ (i.e. $x\in P$). But then $x\in Z(K)$, a contradiction.

So $K$ can only have cyclic factors of order smaller than $p$, unless $H=C_p$ (and in this case there are certainly no cyclic factors larger than $p$).

{\it Case 3: $H=S_4$.}\\
Here, $K$ is the extension of a subgroup of $S_3^{\tilde{n}}$ (not contained in $C_3^{\tilde{n}}$!) by an elementary abelian $2$-group $P$. First we look for cyclic factors of the subgroup $K/P$ of $S_3^{\tilde{n}}=AGL_1(3)^{\tilde{n}}$; but this subgroup has a structure just like the groups considered in the previous case, so we already know that there are no cyclic factors of order larger than 2. So if $K$ had any normal subgroup $N$ with cyclic factor group of order larger than 2, then $N$ could not contain $P$, so $U := K \cap A_4^{\tilde{n}}$ would have a normal subgroup $N\cap U$ (obviously not containing $P$ any more!) with cyclic quotient, i.e. $U' < P$. But one proves $U'=P$ just as in the $AGL_1(p)$-case. 

So we have proven, under the assumptions of ii), that $H=C_p$ (so $K$ is elementary-abelian), and the assertion now follows by induction, because if the factor group $G/K$ contains no element of order larger than its degree, then the analogous statement holds in $G$.
\end{proof}

\section{Proof of Theorem \ref{T1}}
We are now ready for the proof of the main theorem.
\begin{proof}\ \\
{\it First step: Reduction to group theory via ramification theory}\\
We first follow \cite{M} (Proof of theorems 4.9 and 4.6) in reducing the problem to a group theoretic one:\\
$H(t,X)$ is not a proper power over $\overline{\mathbb{Q}}$, therefore by elementary transformations one sees that $H(t,X) - 1$ is absolutely irreducible. Assume there are infinitely many $t_0 \in \mathbb{Z}$ such that $H(t_0,X) - 1$ becomes reducible. Then according to \cite{M}, Prop. 2.1., there is a 
rational function $g(Z)\in \mathbb{Q}(Z)$ such that the following hold:
\begin{itemize}
\item[i)] $g$ is a $\mathbb{Z}$-Siegel function, that is the set $g(\mathbb{Q}) \cap \mathbb{Z}$ is infinite.
\item[ii)] $H(t,X)-1$ is reducible over $\mathbb{Q}(z)$, where $z$ is a root of $g(Z)-t$.
\end{itemize}

Write $H(t,X) - 1 = t^n H_2(\frac{X}{t}) - 1$. Substituting $\frac{X}{t}$ by $X$ and denoting by $z$ a zero of $g(Z)-t$, we get that the polynomial $f(X):= t^n H_2(X) - 1$ is irreducible over $\overline{\mathbb{Q}}(t)$, whereas it becomes reducible over $\overline{\mathbb{Q}}(z)$. Let $x$ be a root of $f$ over $\overline{\mathbb{Q}}(t^n)$, let $\overline{\mathbb{Q}}(y)$ be a minimal intermediate field of $\overline{\mathbb{Q}}(z)|\overline{\mathbb{Q}}(t^n)$ over which $f$ is reducible, and assume without loss that $y$ is contained in the splitting field of $f$ over $\overline{\mathbb{Q}}(t^n)$.

Set $G:=Gal(f|\overline{\qq}(t^n))$. $G$ acts on the roots of $f$ as a transitive subgroup of $S_n$.
Let $\tau \in G$ be the generator of an inertia subgroup of a place extending $t^n\mapsto 0$.
Similarly, let $\sigma\in G$ be the generator of an inertia subgroup of a place extending $t^n\mapsto \infty$.\\
As $\frac{1}{H_2(x)} = t^n$, the place $t^n \mapsto 0$ is fully ramified (of ramification index $deg(H_2) = n$) in $\overline{\mathbb{Q}}(x)$, i.e. the corresponding inertia subgroup generator $\tau$ in an $n$-cycle.\\
Furthermore, $H_2$ is not a proper power, so if we denote by $n_1,...,n_r$ the multiplicity of the zeros of $H_2$, we get $\gcd(\{n_1,...,n_r\}) = 1$. The $n_i$ are of course also the cycle lengths of $\sigma$ (see e.g.\ Lemma 3.1 in \cite{M}).\\
\\
Now consider the ramification in $\overline{\mathbb{Q}}(y)|\overline{\mathbb{Q}}(t^n)$. As $g(z)^n = t^n$ and a $\mathbb{Z}$-Siegel function has either one pole or two algebraically conjugate poles (see \cite{Lg}, 8.5.1), the inertia subgroup corresponding to a place of $\overline{\mathbb{Q}}(z)$ lying over $t^n \mapsto \infty$ is generated by an element consisting of at most two cycles of equal length. The same therefore holds for the places of $\overline{\mathbb{Q}}(y)$ lying over $t^n \mapsto \infty$. Let $m$ be the cycle length in the latter field, i.e. the inertia group generator here has cycle structure $(m)$ or $(m,m)$, or in other words, $\sigma$ has cycle structure $(m)$ or $(m,m)$ in the action on $G/G_y$.\\
\\
Let $u$ be an orbit length of the stabilizer $G_x$ in its action on $G/G_y$. As a conjugate of $\sigma^{n_i}$ lies in $G_x$, and $\sigma^{n_i}$ has orbits of length $\frac{m}{(m,n_i)}$, we get that $\frac{m}{(m,n_i)}$ divides $u$, for all $i$. I.e., $m$ divides all $u \cdot (m,n_i)$, and therefore also the greatest common divisor of these terms, which is just $u$, as $\gcd(\{n_1,...,n_r\}) = 1$.
\\
So $u$ is a multiple of $m$, and as $G_x$ has to act intransitively on the conjugates of $y$ (because so does $G_y$ on the conjugates of $x$), we get that $u=m$ and $\sigma$ must have cycle structure $(m,m)$ on $G/G_y$.
That means, the two places of $\overline{\mathbb{Q}}(y)$ over $t^n \mapsto \infty$ are fully ramified in $\overline{\mathbb{Q}}(z)$, and as that is a rational function field, the Riemann-Hurwitz genus formula (cf.\  e.g.\ \cite[Theorem 3.4.13]{St}) yields that those are the only places ramified in $\overline{\mathbb{Q}}(z)|\overline{\mathbb{Q}}(y)$. In particular the places over $t^n \mapsto 0$ are unramified in $\overline{\mathbb{Q}}(z)|\overline{\mathbb{Q}}(y)$, i.e. (as $\overline{\mathbb{Q}}(t)|\overline{\mathbb{Q}}(t^n)$ lies in $\overline{\mathbb{Q}}(z)$) they have ramification index $n$ over $t^n \mapsto 0$. So $2m$ is a multiple of $n$:\\
 $2m = k\cdot n$, with some $k \in \mathbb{N}$.

By transferring the places of $\overline{\mathbb{Q}}(y)$ extending the place $t^n\mapsto\infty$ to $0$ and $\infty$, we can assume $\frac{h(y)^n}{y^m} = t^n$, with a separable polynomial $h$ of degree $k$.

{\it Second step: Application of the results of Section \ref{permgroups}}\\
We will now show, that $m$ cannot be a multiple of $n$.\\
Assume $m$ were a multiple of $n$; then we get $\frac{h(y)}{y^{m/n}} = t$, with $\deg(h) \ge 2$, so $\overline{\mathbb{Q}}(t)$ is a proper subfield of $\overline{\mathbb{Q}}(y)$.\\
In particular, $\overline{\qq}(t)$ is then contained in the splitting field of $f$ over $\overline{\qq}(t^n)$, and of course the extension $\overline{\qq}(t)|\overline{\qq}(t^n)$ is normal with cyclic Galois group of order $n$.\\
 Setting $N := Gal(f|\overline{\mathbb{Q}}(t))$ we therefore get $N\trianglelefteq G$ and $[G:N]=n$.
Furthermore the place $t^n\to 0$ is already fully ramified of ramification index $n$ in $\overline{\qq}(t)$, which means there is no further ramification above $\overline{\qq}$(t); in other words $\langle \tau \rangle\cap N = \{1\}$.

So we have shown $G = N \rtimes \langle\tau\rangle$, where $N$, in the action on $G/G_x$, is a transitive normal subgroup, and $\tau$ is an $n$-cycle in $S_n$. This is however impossible by Theorem \ref{group}i).

So $2m$ must be an {\it odd} multiple of $n$ (and in particular $n$ is even, which was already shown in \cite{M}, Theorem 4.9). Then we have $\frac{h(y)^2}{y^k} = t^2$. Just as in the above case, we get $G = N \langle\tau\rangle$, with a transitive normal subgroup $N$ ($:=Gal(f | \overline{\mathbb{Q}}(t^2))$)); this time with $|N \cap \langle\tau\rangle| = 2$.\\
\\
Theorem \ref{group}ii) shows that $G$ does not contain an element of order $>n$. But $G$ contains the element $\sigma$, which modulo $K_y := core_G(G_y)$ \footnote{By $core_G(U)$ we denote the kernel of the action of $G$ on $G/U$, or equivalently the largest normal subgroup of $G$ contained in $U$} has order $m = k \cdot \frac{n}{2}$, $k$ odd; so $k=1$. But then the inertia group generators of $\overline{\mathbb{Q}}(y)|\overline{\mathbb{Q}}(t^n)$ over $0$ and $\infty$ have cycle structure $(n)$ and $(\frac{n}{2},\frac{n}{2})$ respectively. By the Riemann-Hurwitz genus formula, there can only be one more ramified place in this extension, and it has to be simply ramified, i.e. the inertia group generator is a transposition in $S_n$.

{\it Final step: Showing $G=C_2$ or $G=D_4$}\\
So $G/K_y$ is generated by an $n$-cycle, an $(\frac{n}{2},\frac{n}{2})$-cycle and a transposition, and the product of these three elements is the identity. This readily implies that $G/K_y = C_2 \wr C_{n/2}$, which can be seen as follows:\\
By appropriate numbering, the $((\frac{n}{2},\frac{n}{2}))$-cycle is $(1,3,5,...,n-1)(2,4,6,...,n)$, and the transposition is $(1,2)$. The group generated by these two elements acts imprimitively, with the block system $\{\{1,2\}, \{3,4\},...,\{n-1,n\}\}$. Also, its image in the action on the  $\frac{n}{2}$ blocks is a cyclic group of order $\frac{n}{2}$, as the transposition acts trivially on the blocks. Therefore $G/K_y \le C_2 \wr C_{n/2}$, and the existence of a transposition in this group enforces equality. 
\\
\\
Furthermore $N K_y/K_y = C_2^{n/2}$, the block kernel of the above wreath product.

$G$ also acts transitively on the cosets of $G_xK_y$, with kernel at least $K_y$. But also $G_xK_y$ is still intransitive on $G/G_y$, so $G_y$ is intransitive on $G/(G_xK_y)$. In particular, $G_y\cdot core_G(G_xK_y)$ is intransitive on $G/G_x$, which by minimality of $\overline{\mathbb{Q}}(y)$ enforces $K_y \ge core_G(G_xK_y)$, with equality altogether.\\
Then however $G/K_y$ has a faithful transitive action on $\tilde{n}:=[G:G_xK_y]$ points (note that $\tilde{n}|n$), with $NK_y/K_y = C_2^{n/2}$ acting as a transitive normal subgroup. 

But as a transitive abelian group, this subgroup acts regularly, so $2^{n/2}=\tilde{n}\le n$. This only leaves $n \in \{2,4\}$, and $\tilde{n} = n$, i.e. $G/K_y = C_2$ or $D_4$, and $G_xK_y = G_x$, which yields $K_y = \{1\}$, as the action on $G/G_x$ is of course faithful.\\
Therefore we are left with $G=C_2$ or $G= D_4$. These examples occur indeed, as mentioned after the statement of theorem \ref{T1}.\\
This completes the proof.
\end{proof}
{\bf Remark:}\\
It is easy to write down all rational polynomials with monodromy group $C_2$ or $D_4$ (there is only one possible ramification structure in each case). The above proof then shows that the examples given in the remark after Theorem \ref{T1} are in fact the only counter-examples (up to linear transformations in the variables).


\begin{thebibliography}{9}
\bibitem{F}
W.\ Feit,
\textit{Some Consequences of the Classification of Finite Simple Groups}.
Proc.\ Symp.\ Pure Math., 37, Amer.\ Math.\ Soc.\  (1980), 175--181.
\bibitem{Ga}
W.\ Gasch\"utz,
\textit{Zur Erweiterungstheorie der endlichen Gruppen}.
J.\ reine angew.\ Math.\ 190 (1952), 93--107.
\bibitem{J}
G.A.\ Jones,
\textit{Cyclic Regular Subgroups of Primitive Permutation Groups}.
J.\ Group Theory 5 (2002), 403--407.
\bibitem{Lg}
S.\ Lang, 
\textit{Fundamentals of Diophantine Geometry}.
Springer Verlag, New York (1983).
\bibitem{L}
K.\ Langmann,
\textit{Werteverhalten holomorpher Funktionen auf \"Uberlagerungen und zahlentheoretische Analogien II}.
Math.\ Nachr.\ 211 (2000), 79--108.
\bibitem{M}
P.\ M\"uller,
\textit{Finiteness Results for Hilbert's Irreducibility Theorem}.
Ann.\ Inst.\ Fourier 52 (2002), 983--1015.
\bibitem{St}
H.\ Stichtenoth,
\textit{Algebraic Function Fields and Codes}.
Springer Verlag, GTM 254 (2008).
\end{thebibliography}
\end{document}